\documentclass[12pt,twoside]{article}
  \usepackage{geometry}
\usepackage{amssymb}
\usepackage{amsmath}
\usepackage{amsthm}
\usepackage{mathrsfs}
\usepackage{color}
\usepackage{changes}
\usepackage[colorlinks,
linkcolor=blue,
anchorcolor=blue,
citecolor=red
]{hyperref}

 \def\rr{\mathbb{R}}
\def\zz{\mathbb{Z}}
\def\nn{\mathbb{N}}
\def\cf{\mathcal{F}}

\def\cs{\mathcal{S}}
\def\cc{\mathbb{C}}

\def\supp{{\rm{\ supp\ }}}

\def\rn{\mathbb{R}^n}
\def\zn{\mathbb{Z}^n}

\def\les{\lesssim}

\newtheorem{thm}{Theorem}[section]

\newtheorem{lem}[thm]{Lemma}
\newtheorem{prop}[thm]{Proposition}
\newtheorem{cor}[thm]{Corollary}
\newtheorem{defn}[thm]{Definition}

\def\XXint#1#2#3{{
\setbox0=\hbox{$#1{#2#3}{\int}$}
\vcenter{\hbox{$#2#3$}}\kern-.5\wd0}}

\allowdisplaybreaks[4]
\numberwithin{equation}{section}

\begin{document}
\title{Embedding and Duality of  Matrix-Weighted Modulation Spaces \footnotetext{$\ast$ The corresponding author J. S. Xu  jingshixu@126.com\\
The work is supported by the National Natural Science Foundation of China (Grant No. 12161022) and the Science and Technology Project of Guangxi (Guike AD23023002).}}
\author{Shengrong Wang\textsuperscript{a}, Pengfei Guo\textsuperscript{a}, Jingshi Xu\textsuperscript{b,c,d*}\\
{\scriptsize  \textsuperscript{a}School of Mathematics and Statistics, Hainan Normal University, Haikou, 571158, China}\\
{\scriptsize  \textsuperscript{b}School of Mathematics and Computing Science, Guilin University of Electronic Technology, Guilin 541004, China} \\
{\scriptsize  \textsuperscript{c} Center for Applied Mathematics of Guangxi (GUET), Guilin 541004, China}\\
{\scriptsize  \textsuperscript{d}Guangxi Colleges and Universities Key Laboratory of Data Analysis and Computation, Guilin 541004, China}}
\date{}

\maketitle

{\bf Abstract.} In this paper, we give an approximation characterization,  embedding properties and the duality of matrix weighted modulation spaces.



\section{Introduction}
Recent decades, matrix weighted function spaces have attracted many author's attention.
Indeed, Treil and Volberg \cite{mwntv1} introduced the Muckenhoupt $\mathcal{A}_2$ matrix weight and generalized the Hunt-Muckenhoutt-Wheeden theorem to the vector-valued case.
Nazarov and Treil then generalized the results to matrix $\mathcal{A}_p$ weights for $p \in (1,\infty)$, and Volberg extended the theory of weighted norm inequalities on $L^p$ to the case of vector-valued functions in \cite{mwntv3}.
Goldberg \cite{mwg1} showed that the matrix $\mathcal{A}_p$ condition leads to $L^p$-boundedness of the Hardy-Littlewood
maximal operator, and used this estimate to establish a bound for the weighted $L^p$ norm of singular integral operators. When $W\in \mathcal{A}_p$,
Roudenko \cite{mwr1} showed that the dual of $\dot{B}^\alpha_{p,q}(W)$ can be identified with $\dot{B}^{-\alpha}_{p^{\prime},q^{\prime}}(W^{-p^{\prime}/p})$ for $\alpha \in \rr$, $0<q<\infty$ and $1< p<\infty$, where $q'=q/(q-1)$ if $q\in (1,\infty)$ and $q'=1$ if $q\in (0,1]$.
Cruz-Uribe, Isralowitz and Moen \cite{mcim1} extended the theory of two weight, $\mathcal{A}_p$ bump conditions to the setting of matrix weights and proved two matrix weight inequalities for
fractional maximal operators, fractional and singular integrals, sparse operators and averaging operators.
Wang, Yang and Zhang \cite{wyz-1} characterized the matrix-weighted Triebel-Lizorkin space $\dot{F}^{\alpha,q}_p(W)$ by the Peetre maxima function, the Lusin area function, and the Littlewood-Paley $g^\ast_\lambda$ function.
As an application, the boundedness of the Fourier multiplier in the matrix-weighted Triebel-Lizorkin space is given.
Bu, Yang and Yuan \cite{byy-1} introduced homogeneous $(\mathcal{X},d,\mu)$ matrix-weighted Besov spaces in the sense of Coifman and Weiss, and proved that matrix-weighted Besov spaces are independent of the approximation of exponential decay identities and the choice of distribution spaces.
Moreover, they obtained the wavelet characterization and molecular characterization of the matrix-weighted Besov space, the boundedness of the almost diagonal operators on the matrix-weighted sequence Besov spaces, and the boundedness of the Calder\'{o}n-Zygmund operator on  matrix-weighted Besov spaces.
Bu, Hyt\"{o}nen, Yang and Yuan \cite{bhyy-1,bhyy-2,bhyy-3,bhyy-4} introduced the new concept of $\mathcal{A}_p$ dimensions of matrix weights and delved into their properties. Then they introduced spaces $\dot{B}^{s,\tau}_{p,q}(W)$ and $\dot{F}^{s,\tau}_{p,q}(W)$, and obtained their characterizations by $\varphi$-transform, molecule and wavelet, and the optimal boundedness of the pseudo-differential operators and the Calder\'{o}n-Zygmund operators on these spaces.

On the other hand, since modulation spaces were introduced by Feichtinger \cite{mwf1} in 1983, they have also attracted many author's attention.
In fact, Kobayashi and Sugimoto \cite{mwmm1} clearly determined the inclusion relationship between $L^p$-Sobolev space and modulation space.
Zhao, Gao and Guo \cite{zgg1} gave the optimal embedding relations between local Hardy space and $\alpha$-modulation spaces.
Sawano \cite{mws1} presented a natural extension of the modulation spaces $M^s_{p,q}(w)$ with $w \in A^{\rm loc}_\infty$, and investigated their atomic and molecular decomposition as well as some elementary properties.
Guo, Fan and Zhao \cite{mwgfz1} considered the embedding relations between any two $\alpha$-modulation spaces.
Sawano \cite{mws2} considered atomic decomposition of $M^s_{p,q}$ with $0<p,q \leq \infty$, $s \in \rr$.
Han and Wang \cite{mwhw1} considered some fundamental properties including dual spaces, complex interpolations of $\alpha$-modulation spaces $M^{s,\alpha}_{p,q}$ with $0 < p,q \leq \infty$, and obtained necessary and sufficient conditions for the scaling property and the inclusions between $\alpha_1$-modulation and $\alpha_2$-modulation spaces.
Chen, Lu and Wang \cite{mwclw1} studied the embedding properties of the scaling limit of the modulation spaces, including the homogeneous case and non-homogeneous case.
Recently, in \cite{nm-1} Nielsen introduced matrix weighted $\alpha$-modulation spaces $M^{s,\alpha}_{p,q}(W)$ and discrete matrix weighted $\alpha$-modulation spaces $m^{s,\alpha}_{p,q}(W)$, and proved their equivalence by using an adaptive compact frame.
Then, the boundedness of almost diagonal operators on these spaces was given, and then the molecular characterization was given.
The boundedness of Fourier multipliers on matrix weighted $\alpha$-modulation spaces $M^{s,\alpha}_{p,q}(W)$ was also given in \cite{nm-1}.

Inspirited by the above mentioned works, in this paper, we consider an approximate characterization,  the  embedding properties and the duality of matrix weighted modulation spaces.
The plan of the paper is as follows. In Section \ref{mwm-s1}, we collect some notations.
In Section \ref{mwm-s2}, we give the connection between averaging matrix-weighted modulation spaces and matrix-weighted modulation spaces.
In Section \ref{mwm-s3}, we give equivalent norms and the approximate characterization of these spaces.
In Section \ref{mwm-s4}, we obtain the embedding properties of these spaces.
In Section \ref{mwm-s6}, we obtain the duality of these spaces.

\section{Preliminaries}\label{mwm-s1}

In this section, we recall some definitions and concepts.  First, we make some convention.

Let $\nn$ be the collection of all natural numbers and $\nn_0 = \nn \cup \{0\}$.
Let $\zz$ be the collection of all integers.
Let $\rn$ be $n$-dimensional Euclidean  space, where $n \in \nn$. Put $\rr= \rr^1$, whereas $\mathbb{C}$ is the complex plane.
In the sequel, $C$ denotes positive constants, but it may change from line to line.
For any quantities $A$ and $B$, if there exists a constant $C>0$ such that $A\leq CB$, we write $A \lesssim B$. If $A\lesssim B$ and $B \lesssim A$, we write $A \sim B$.
$B(x,r)$ denoting Ball $B$ of $\rn$  of radius $r>0$ centered at $x \in \rn$.
Let $Q_0=\{\xi : \xi_i \in [-1/2,1/2), i=1,\cdots,n\}$ and $Q_k=k+Q_0, k\in\zn$.
It is easily to see that $\{Q_k\}_{k \in \zn}$ consists in a unit-cube decomposition of $\rn$, that means $\bigcup_{k \in \zn}Q_k=\rn$ and $Q_k \cap Q_j=\emptyset$ if $k\neq j$.
Let us define $\ell + A := \{\ell + a : a\in A\}$ for $\ell \in \zn$ and $A\subseteq \rn$.

Let $p \in [1,\infty)$. Then the Lebesgue space  $L^{p}(\rn)$  equipped with the norm
\[\|f\|_{L^p(\rn)} = \bigg( \int_{\rn} |f(x)|^p {\rm d}x \bigg)^{1/p}.\]
The space $L^{p} _{\rm loc}(\rn)$ is defined by
$L^{p} _{\rm loc}(\rn):=\{f: f\chi_K \in L^{p}(\rn)$  for all compact subsets  $K \subset \rn\},$
where and what follows, $\chi_{S}$ denotes the characteristic function of a set $S\subset \rn$.

The set $\mathcal{S}(\rn)$ stands for the usual Schwartz space of rapidly decreasing complex-valued functions and $\mathcal{S}'(\rn)$ the dual space of tempered distributions.
For $f \in \mathcal{S}(\rn)$, let $\mathcal{F} f$ or $\widehat{f}$ denote the Fourier transform of $f$ defined by
\[ \mathcal{F} f(\xi) =\widehat{f}(\xi) :=(2 \pi)^{-n / 2} \int_{\rn} e^{-i x \xi} f(x) {\rm d}x , \ \xi \in \rn\]
while $f^\vee(\xi)= \widehat{f} (-\xi)$ denote the inverse Fourier transform of $f$.

Now we recall some basic matrix concepts.
For any $m\in \nn$, $M_m(\cc)$ is denoted as the set of all $m \times m$ complex-valued matrices.
For any $A \in M_m(\cc)$, let
\[\|A\|:= \sup_{|\vec{z}|=1} |A\vec{z}|,\]
where $\vec{z}:= (z_1,...,z_m)^{\rm T} \in \cc^m$ and $|\vec{z}| =: (\sum_{i=1}^m |z_i|^2)^{1/2}$, T denotes the transpose of the row vector.

Diagonal matrix $A$ can be denoted as $A={\rm diag}(\lambda_1,...,\lambda_m)$, where $\{\lambda_i\}_{i=1}^m \subset \rr$ .
If $\lambda_1=\cdots =\lambda_m=1$ in the diagonal matrix above, it is called the identity matrix and is denoted by $I_m$.
If there is a matrix $A^{-1} \in M_m(\cc)$ such that $A^{-1}A = I_m$, then matrix $A$ is said to be invertible.

A matrix $A  \in M_m(\cc)$ is called positive definite, if for any $\vec{z}\in \cc^m \setminus \{\vec{0}\}$, $(A\vec{z},\vec{z}) >0$. $A  \in M_m(\cc)$ is nonnegative positive definite, if for any $\vec{z}\in \cc^m \setminus \{\vec{0}\}$, $(A\vec{z},\vec{z}) \geq 0$.

Next, we recall the concept of scale weights and matrix weights.
\begin{defn}
Fix $p \in (1,\infty)$. A positive measurable function $w$ is said to be in the Muckenhoupt class $A_p$, if there exists a positive constant $C$ such that, for all balls $B$ in $\mathbb{R}^n$, such that
\[  \bigg(\frac{1}{|B|}\int_B w(x){\rm d}x \bigg) \bigg(\frac{1}{|B|}\int_B w(x)^{1-p^{\prime}}{\rm d}x \bigg)^{p-1} \leq C.\]
\end{defn}
We say $w \in A_1$, if $Mw(x) \leq Cw(x)$ for a.e. $x$. If $1 \leq p < q < \infty$, then $A_p \subset A_q$.
We denote $A_\infty =\cup_{p>1} A_p$.
Let $1<p < \infty$ and $w$ be a (scalar) weight which is a nonnegative measurable function on $\rn$.  We define the weighted $L^p(w)$ space, which is a Banach space equipped with the norm
\[\|f\|_{L^p(w)} = \bigg( \int_{\rn} |f(x)|^p w(x) {\rm d}x \bigg)^{1/p}.\]

\begin{defn}
A matrix-valued function $W:\rn \rightarrow M_m(\cc)$ is called a matrix weight if $W$ satisfies that\\
{\rm (i)} for any $x \in \rn, W(x)$ is nonnegative definite;\\
{\rm (ii)} for almost every $x \in \rn$, $W(x)$ is invertible;\\
{\rm (iii)} the entries of $W$ are all locally integrable.
\end{defn}

The following definition from \cite{mwr3}.
\begin{defn}
Let $W$ be a matrix weight, $1<p<\infty$. Then  $W \in \mathcal{A}_p$ if and only if
\[ \sup_B \frac{1}{|B|} \int_B \bigg( \frac{1}{|B|} \int_B \| W^{1/p}(x) W^{-1/p}(y) \|^{p^{\prime}} {\rm d}y \bigg)^{p/p^{\prime}} {\rm d}x < \infty, \]
where the supremum is taken over all balls $B \subset \rn$.
We say that $W \in \mathcal{A}_1$, if
\[ \sup_{B}  \mathop{ {\rm ess} \sup}_{y \in B} \frac{1}{|B|} \int_B \|W(t) W^{-1}(y)\| {\rm d}t< \infty, \]
where $\|\cdot\|$ denotes the operator norm of a matrix.
\end{defn}

\begin{defn}
Let $1 \leq p<\infty$ and  $W$ be a matrix weight. The matrix-weighted
Lebesgue space $L^p(W)$ is defined to be the set of all  measurable  vector-valued functions $\vec{f}=(f_1, \cdots, f_m)^{\rm T}: \rn \rightarrow \cc^m$ such that
\[ \|\vec{f} \|_{L^p(W)}= \bigg(\int_{\rn} \|W^{1/p}(x) \vec{f}(x) \|^p {\rm d}x \bigg)^{1/p}<\infty.\]
\end{defn}

\begin{defn}\label{mwm-D2}
Let $1 \leq p < \infty$, a matrix weight $W$ is called doubling matrix weight of order $p$, if there exists a constant $C$ such that for any cube $Q\in \rn$ and any $\vec{z} \in \rn$
\begin{equation}\label{wbl-17}
 \int_{2Q} | W^{1/p}(x) \vec{z}|^p {\rm d}x \leq C  \int_{Q} | W^{1/p}(x) \vec{z}|^p {\rm d}x .
\end{equation}
If $C = 2^\beta$ is the smallest constant for which (\ref{wbl-17}) holds, then $\beta$ is called the doubling exponent of $W$.
\end{defn}

To recall the definition of modulation spaces, we need some general definitions.
Let $\rho \in \cs(\rn)$ and $\rho: \rn \rightarrow [0,1]$ be a smooth radial dump function adapted to the ball $B(0,\sqrt{n})$, satisfying $\rho(\xi)=1$ for $|\xi| \leq \sqrt{n}/2$ and $\rho(\xi)=0$ for $|\xi| \geq \sqrt{n}$.
Let $\rho_k$ be a translation of $\rho: \rho_k(\xi) = \rho(\xi - k)$, $k \in \zn$. Thus, we observe that $\rho_k =1$ in $Q_k$ and $\sum_{k \in \zn} \rho_k(\xi) \geq 1$ for all $\xi \in \rn$.
We write
\[ \phi_k(\xi)=\rho_k(\xi) \bigg( \sum_{k \in \zn} \rho_k(\xi) \bigg)^{-1}, \quad k \in \zn. \]
Then, we have
\begin{equation}\label{wm-4}
\begin{cases}
|\phi_k(\xi)| \geq c, \quad \forall \xi \in Q_k,\\
\supp \phi_k \subset \{ \xi: |\xi-k| \leq \sqrt{n}\},\\
\sum_{k \in \zn} \phi_k(\xi) \equiv 1, \quad \forall \xi \in \rn,\\
|D^\alpha \phi_k(\xi)| \leq C_\beta, \quad \forall \xi \in \rn, |\alpha| \leq \beta.
\end{cases}
\end{equation}
Therefore, the set
\[ Y:=\big\{ \{\phi_k\}_{k \in \zn}: \{\phi_k\}_{k \in \zn} \text{ satisfies (\ref{wm-4})} \big\} \]
is nonempty. Let $\{\phi_k\}_{k \in \zn} \in Y$ be a function sequence. We define
\[ \square_k := \cf^{-1} \phi_k \cf, \quad k \in \zn, \]
which are said to be frequency-uniform decomposition operators. For any $k \in \zn$, we write $\langle k \rangle =(1+|k|^2)^{1/2}$, where $|k|=|k_1|+|k_2|+\cdots +|k_n|$.

In what follows, for any $m \in \nn$, let
\[ [\cs(\rn)]^m:= \{ \vec{f}:=(f_1,\ldots,f_m)^{\rm T}: \text{ for any } i \in \{1,\ldots,m\}, f_i \in \cs(\rn)  \}\]
and
\[ [\cs'(\rn)]^m:= \{ \vec{f}:=(f_1,\ldots,f_m)^{\rm T}: \text{ for any } i \in \{1,\ldots,m\}, f_i \in \cs'(\rn)  \}.\]
\begin{defn}\label{mwm-D1}
Let $1 \leq  p< \infty$, $0<q \leq \infty$ and $s\in \rr$, and $W$ be a matrix weight.
The matrix weighted modulation space $M^{s}_{p,q}  (W)$ is the collection of $\vec{f} \in [\cs'(\rn)]^m$ such that
\[\|\vec{f}\|_{M^{s}_{p,q}(W) }:=   \bigg( \sum_{k \in \zn} \big(\langle k \rangle^s \|W^{1/p} \square_k \vec{f} \|_{L^p} \big)^q \bigg)^{1/q}    <\infty.\]
\end{defn}

\begin{lem}[see {\cite[Remark 2.15]{nm-1}}]\label{wm-L1}
Let $1 \leq p < \infty$, $W$ be a matrix $\mathcal{A}_p$ and $\{\phi_k\}_{k \in \zn} \in Y$. Then there exists a constant $C>0$ such that
\[ \|\cf^{-1} \phi_k \cf \vec{f}\|_{L^p(W)} \leq  C \|\vec{f}\|_{L^p(W)}\]
holds for all $\vec{f} \in L^p(W)$.
\end{lem}

\section{Averaging operators}\label{mwm-s2}
In this section, we show that in Definition \ref{mwm-D1}, the matrix $W$ can be replaced by a sequence of averaging operators.
Let $k$, $\ell \in \zn$, $r_k=\langle k \rangle$ and $a \geq \sqrt{n}/2$, we denote $Q(k,\ell)= (a r_k)^{-1} \ell + [0,(a r_k)^{-1})^n$, $x_{Q(k,\ell)}:= (a r_k)^{-1} \ell$.
For fixed $k$, $\mathcal{Q}_k:=\cup_{\ell} Q(k,\ell)$ forms a partition of $\rn$.
Let $\mathcal{Q}:= \{Q(k,\ell): k \in \zn, \ell \in \zn\} $. $B_k:=B(k \langle k \rangle,\sqrt{n} \langle k \rangle)$, $k\in \zn$.
For $k \in \zz$, we denote
\[ \Omega_k:=\{\vec{f}: f_i \in \cs'(\rn) \text{ with} \supp \widehat{f}_i \subseteq B_k, i=1,2,\cdots,m \}. \]
\begin{defn}
Let matrix weight $W$: $\rn \rightarrow M_m(\cc)$, $m \in \nn$, and $1 \leq p < \infty$.  A sequence $\{A_Q\}_{Q \in \mathcal{Q}}$ of reducing operators of order $p$ for $W$ if  for any $z \in \cc^m$ and $Q \in \mathcal{Q}$,
\[|A_Q \vec{z}| \sim \bigg( \frac{1}{|Q|} \int_Q |W^{1/p}(x) \vec{z}|^p {\rm d}x \bigg)^{1/p},\]
where equivalence constants   depend only on $m$ and $p$.
\end{defn}

\begin{defn}
Let $\{A_Q\}_{Q \in \mathcal{Q}}$ be a sequence of positive definite matrices, $\beta \in (0,\infty)$, and $p \in (0,\infty)$.
The sequence $\{A_Q\}_{Q\in \mathcal{Q}}$ is said to be strongly doubling of order $(\beta, p)$ if there exists a positive constant $C$ such that, for any $Q,$ $P \in \mathcal{Q}$,
\[ \|A_QA_P\|^p \leq C \max \bigg\{ \bigg[\frac{\ell(P)}{\ell(Q)} \bigg]^n, \frac{\ell(Q)}{\ell(P)} \bigg]^{\beta-n}  \bigg\} \bigg[ 1+ \frac{|x_Q-x_P|}{\max\{\ell(P),\ell(Q)\}} \bigg]^\beta .\]
\end{defn}

\begin{defn}
Let $k\in \zn$, $1\leq p<\infty$ and $\{A_Q\}_{Q\in \mathcal{Q}}$ is a sequence of reducing operators of order p for $W$.
The  space $L^p(\{A_{Q}\},k)$ is the collection of $\vec{f} \in [\cs'(\rn)]^m$ such that
\[\|\vec{f}\|_{L^p(\{A_{Q}\},k)} = \bigg\|\sum_{\ell(Q)=(ar_k)^{-1}}  A_{Q} \vec{f} \chi_{Q} \bigg\|_{L^p}<\infty.\]
\end{defn}

\begin{defn}\label{mwm-D3}
Let $1 \leq  p< \infty$, $0<q \leq \infty$ and $s\in \rr$, and $\{A_Q\}_{Q \in \mathcal{Q}}$ be a sequence of $m \times m$ positive definite matrices.
The $\{A_Q\}$-modulation space $M^{s}_{p,q}(\{A_Q\})$ is the collection of $\vec{f} \in [\cs'(\rn)]^m$ such that
\[\|\vec{f}\|_{M^{s}_{p,q}(\{A_{Q}\}) }:=   \bigg( \sum_{k \in \zn} \bigg(\langle k \rangle^s \bigg\|\sum_{\ell(Q)=(ar_k)^{-1}} A_Q \square_k \vec{f} \chi_Q \bigg\|_{L^p} \bigg)^q \bigg)^{1/q}    <\infty.\]
\end{defn}

\begin{lem}[see {\cite[Lemma 4.7]{nm-1}}]\label{mwm-L3}
Let $W$ be a doubling matrix weight of order $p > 0$ with doubling exponent $\beta$
as specified in Definition \ref{mwm-D2}.
Let $\{A_Q\}_{Q \in \mathcal{Q}}$ is a sequence of reducing operators of
order $p$ for $W$.
Then $\{A_Q\}_{Q \in \mathcal{Q}}$ is strongly doubling of order $(\beta, p)$.
\end{lem}

\begin{lem}\label{wm-L8}
Let $1\leq p< \infty$, $W$ be a doubling matrix weight of order $p$, and $\vec{f} \in \Omega_k$, $k \in \zz$. Then there exists a constant $C>0$ independent of $k$ such that
\begin{equation}\label{wm-9}
\|\vec{f}\|_{L^p(W)} \leq C \|\vec{f}\|_{L^p(\{A_Q\},k)},
\end{equation}
where $\{A_Q\}_{Q \in \mathcal{Q}}$ is a sequence of reducing operators of order $p$ associated with $W$.
\end{lem}

 \begin{proof} We use the method in the proof of  {\cite[Lemma 5.1]{mwr1}}.
Set $\vec{h}_k(x)=e^{ik \cdot \frac{x}{a} }\vec{f}(\frac{x}{ar_k})$ for $k \in \zn$, where $a \geq \sqrt{n}/2$.
Since $\vec{f}\in \Omega_k$, then it is easy to see $\supp (\hat{\vec{h}}_k)  \subseteq B(0,2)$.
By changing the variables, we have
\[\|\vec{f}\|_{L^p(W)}^p = \int_{\rn} |W^{1/p}(x) \vec{f}(x)|^p {\rm d}x \sim r_k^{-n} \sum_{\ell \in \zn} \int_{\ell+[0,1)^n} \Big|W^{1/p}\Big(\frac{x}{ar_k}\Big) \vec{h}_k (x) \Big) \Big|^p {\rm d}x.\]
Let $\Gamma = \{\gamma \in \cs(\rn): \widehat{\gamma}=1 \text{ on } \{\xi \in \rn: |\xi| \leq 2\} \text{ and } \supp \widehat{\gamma} \subseteq \{\xi \in \rn : |\xi|<\pi\} \}$.
Thus, there exists $\gamma \in \Gamma$ such that $\vec{h}_k  =\vec{h}_k \ast \gamma$. Using the decay of $\gamma$ and H\"{o}lder's inequality, we obtain
\[\|\vec{f}\|_{L^p(W)}^p \les   r_k^{-n} \sum_{\ell \in \zn} \int_{\ell+[0,1)^n}   \sum_{\nu \in \zn} \int_{\nu+[0,1)^n} \frac{\big|W^{1/p}\big(\frac{x}{ar_k}\big) \vec{h}_k (y) \big) \big|^p}{(1+|\nu-\ell|)^{M}} {\rm d}y {\rm d}x \]
for some $M > \beta+n$. Note that
\[\int_{Q(k,\ell)} |W^{1/p}(x) \vec{h}_k (y)|^p {\rm d}x \sim r_k^{-n}\int_{\ell+[0,1)^n} |W^{1/p}\big(\frac{x}{ar_k}\big) \vec{h}_k (y)|^p {\rm d}x,\]
which mean $|A_{Q(k,\ell)} \vec{h}_k(y)|^p \sim  \int_{\ell+[0,1)^n} |W^{1/p}\big(\frac{x}{ar_k}\big) \vec{h}_k (y)|^p {\rm d}x$.
Using the doubling property of $W$ (Lemma \ref{mwm-L3} and Definition \ref{mwm-D3}) to shift $Q(k,\ell)$ to $Q(k,\nu)$,   we have
\begin{align*}
\|\vec{f}\|_{L^p}^p
& \les r_k^{-n} \sum_{\ell \in \zn}  \sum_{\nu \in \zn} \int_{\nu+[0,1)^n} (1+|\nu-\ell|)^{-(M-\beta)} |A_{Q(k,\nu)} \vec{h}_k(y)|^p {\rm d}y \\
& \les \sum_{\nu \in \zn} \int_{\nu+[0,1)^n} |A_{Q(k,\nu)} \vec{h}_k(y)|^p {\rm d}y,
\end{align*}
where the sum on $\ell$ converges since $M>\beta+n$. Changing variables $x = (ar_k)^{-1}y$, we obtain (\ref{wm-9}).
\end{proof}

\begin{cor}\label{wm-C1}
Let $s \in \rr$, $1\leq p< \infty$, $0<q< \infty$ and $W$ be a doubling matrix weight of order $p$. Then there exists a constant $C>0$ independent of $k$ such that
\[M^s_{p,q}(\{A_Q\}) \subseteq M^s_{p,q}(W),\]
where $\{A_Q\}_{Q \in \mathcal{Q}}$ is a sequence of reducing operators associated with $W$.
\end{cor}

\begin{proof}
Since $\{\phi_k\}_{k \in \zn} \in Y$, we know that  $\square_k \vec{f} \in \Omega_k$. Taking the $\ell^q$ norms on both sides of the inequality in Lemma \ref{wm-L8}, we obtain
\begin{align*}
\|\vec{f}\|_{M^s_{p,q}(W)}
&= \|\{\langle k \rangle^s \| \square_k \vec{f} \|_{L^p(W)}\}_{k \in \zn}\|_{\ell^q}  \\
&\leq C \|\{\langle k \rangle^s \| \square_k \vec{f} \|_{L^p(\{A_Q\},k)}\}_{k \in \zn}\|_{\ell^q}\\
&= C \|\vec{f}\|_{M^s_{p,q}(\{A_Q\})}.
\end{align*}
 This completes the proof.
\end{proof}

\begin{lem}\label{wm-L11}
Let $1\leq p< \infty$, $W$ be a doubling matrix weight of order $p$. Then there exists a constant $C>0$ independent of $k$ such that
\[ \|\vec{f}\|_{L^p(\{A_Q\},k)} \leq C \|\vec{f}\|_{L^p(W)},\]
where $\{A_Q\}_{Q \in \mathcal{Q}}$ is a sequence of reducing operators associated with $W$.
\end{lem}

\begin{proof}
Set $\vec{h}_k(x)=e^{ik \cdot \frac{x}{a} }\vec{f}(\frac{x}{ar_k})$ for $k \in \zn$, where $a \geq \sqrt{n}/2$. Since $\vec{f}\in \Omega_k$, then it is easy to see $\supp (\hat{\vec{h}}_k)  \subseteq B(0,2)$. Thus,  similar to the proof of  {\cite[Lemma 5.3]{mwr1}},
\begin{align*}
\|\vec{f}\|_{L^p(\{A_Q\},k)}^p
&\sim \sum_{\ell \in \zn} \int_{Q(k,\ell)}  \frac{1}{|Q(k,\ell)|} \int_{Q(k,\ell)} |W^{1/p}(y) f(x)|^p {\rm d}y  {\rm d}x \\
&\sim \sum_{\ell \in \zn} \int_{\ell +[0,1)^n}   \int_{Q(k,\ell)} |W^{1/p}(y) h_k(x)|^p {\rm d}y  {\rm d}x \\
&\les r_k^{-n}\|\vec{h}_k\|_{L^p(W(\frac{\cdot}{ar_k}))}^p\\
&\les \|\vec{f}\|_{L^p(W)}^p,
\end{align*}
where we used {\cite[Lemma 3.2]{nm-1}} instead of {\cite[Lemma 6.3]{mwr1}}.
\end{proof}

\begin{cor}\label{wm-C2}
Let $s \in \rr$, $1\leq p< \infty$, $0<q<\infty$ and $W$ be a doubling matrix weight of order $p$. Then
\[M^s_{p,q}(W)  \subseteq M^s_{p,q}(\{A_Q\}),\]
where $\{A_Q\}_{Q \in \mathcal{Q}}$ is a sequence of reducing operators associated with $W$.
\end{cor}

\begin{proof}
The proof is the same as the proof of Corollary \ref{wm-C1}, replacing Lemma \ref{wm-L8} by Lemma \ref{wm-L11}.
\end{proof}

Combining Corollary \ref{wm-C1} and \ref{wm-C2}, we have the following lemma.
\begin{lem}\label{mwm-L2}
Let $s \in \rr$, $1\leq p< \infty$, $0<q<\infty$, $W$ be a doubling matrix weight of order $p$, and $\{A_Q\}_{Q \in \mathcal{Q}}$ is a sequence of reducing operators associated with $W$. Then
\[M^s_{p,q}(W)  =  M^s_{p,q}(\{A_Q\}).\]
\end{lem}

\section{Equivalent norms}\label{mwm-s3}
In this section, we discuss the norm equivalence of $M^s_{p,q}(W)$, as well as its approximate characterization.
In this section and the following sections, if there is no explanation, we always use the following definition. Define
\begin{equation}\label{wm-5}
\Lambda:=\{ l \in \zn: B(l,\sqrt{2n}) \cap B(0,\sqrt{2n})\neq \emptyset \}.
\end{equation}
\begin{defn}
Let $g \in \cs(\rn)$ be a fixed, nonzero window function (smooth cut-off function) and $\vec{f} \in [\cs'(\rn)]^m$ a tempered distribution.
The short-time Fourier transform of a function $\vec{f}$ with respect to $g$ is defined as
\[ V_g \vec{f}(x,\xi) = \int_{\rn} e^{-it\cdot \xi} \overline{g(y-x)} \vec{f}(y){\rm d}y.\]
\end{defn}

\begin{defn}
Let $1 \leq  p< \infty$, $0<q \leq \infty$ and $s\in \rr$.
Furthermore, let $W$ be a matrix weight and nonzero function $g\in \cs(\rn)$.
We write  for $\vec{f} \in [\cs'(\rn)]^m$,
\[\|\vec{f}\|_{M^{s}_{p,q}(W)}^\circ:=   \bigg( \int_{\rn} \bigg( \int_{\rn}  |W^{1/p} V_g \vec{f}(x,\xi) |^p {\rm d}x \bigg)^{q/p} \langle \xi \rangle^{sq} d\xi \bigg)^{1/q}    <\infty,\]
where $V_g \vec{f}=(V_g f_1, V_g f_2, \cdots, V_g f_m )^{\rm T}$.
\end{defn}

\begin{lem}\label{wm-L2}
Let $\{ \phi_k\}_{k \in \zn}$, $\{ \varphi_k\}_{k \in \zn} \in Y$, $s \in \rr$, $1 \leq p <\infty$, $0<q<\infty$ and $W$ be a matrix in $\mathcal{A}_p$. Then $\{ \phi_k\}_{k \in \zn}$ and $\{ \varphi_k\}_{k \in \zn} \in Y$ generate equivalent norms on $M^s_{p,q}(W)$.
\end{lem}
\begin{proof}
We put
\[ \square^\phi_k :=\cf^{-1} \phi_k \cf, \quad  \square^\varphi_k := \cf^{-1} \varphi_k \cf.\]
By translation identity, we have
\begin{equation}\label{wm-6}
(\cf^{-1} m \cf \vec{f})(x) = e^{ixk} \big[ \cf^{-1} m(\cdot+k) \cf(e^{iky} \vec{f}(y)) \big](x).
\end{equation}
From (\ref{wm-5}), we have
\begin{equation}\label{wm-7}
\|\square^\phi_k \vec{f}\|_{L^p(W)}  \leq \sum_{l \in \Lambda} \| \square^\phi_k \vec{f} \square^\varphi_{k+l}  \|_{L^p(W)}.
\end{equation}
By (\ref{wm-6}) and Lemma \ref{wm-L1} for $\Omega=\cup_{l \in \Lambda} B(l,\sqrt{2n})$, we obtain
\begin{align}\label{wm-8}
 \| \square^\phi_k \vec{f} \square^\varphi_{k+l} \|_{L^p(W)}  &=  \| \cf^{-1} \phi_k (\cdot +k) \varphi_{k+l}(\cdot+k) \cf(e^{-iky} \vec{f}(y)) \|_{L^p(W)}  \nonumber\\
 &\leq  \| \square^\varphi_{k+l} \vec{f} \|_{L^p(W)}.
\end{align}
 Therefore, (\ref{wm-7}) and (\ref{wm-8}) imply that
\[ \|\square^\phi_k \vec{f}\|_{L^p(W)}  \leq \sum_{l \in \Lambda} \| \square^\varphi_{k+l} \vec{f} \|_{L^p(W)}. \]
From the definition of the matrix weighted modulation space, we obtain $\|\vec{f}\|^\phi_{M^s_{p,q}(W)} \lesssim \|\vec{f}\|^\varphi_{M^s_{p,q}(W)}$. The reverse inequality is similar, we end the proof.
\end{proof}

\begin{thm}
Let $1 \leq  p< \infty$, $0<q \leq \infty$, $s \in \rr$ and $W$ be a matrix in $\mathcal{A}_p$. Then there exist positive constants $C_1$ and $C_2$ such that
\[ C_1 \|f\|^\circ_{M^{s}_{p,q}(W)} \leq \|f\|_{M^{s}_{p,q}(W)} \leq C_2 \|f\|^\circ_{M^{s}_{p,q}(W)} .\]
\end{thm}

\begin{proof}
By the fundamental STFT identity (Proposition 1.82 in \cite{mwbo1}), we obtain
\[ V_g \vec{f}(x,\xi)= e^{-2 \pi ix\cdot \xi} V_{\widehat{g}} \widehat{\vec{f}}(\xi,-x)= e^{-2 \pi ix\cdot \xi} \big( \cf^{-1} \widehat{g}(\cdot- \xi) \cf \vec{f} \big)(x). \]
By the mean value theorem, we can find $\xi_k \in Q_k$ such that
\begin{align}\label{wm-1}
\|f\|^\circ_{M^{s}_{p,q}(W)} &= \bigg( \int_{\rn} \langle \xi \rangle^{sq} \big\| \cf^{-1} \widehat{g}(\cdot- \xi) \cf \vec{f} \big\|_{L^p(W)}^q d\xi \bigg)^{1/q} \nonumber \\
& \sim \bigg( \sum_{k \in \zn} \langle k \rangle^{sq} \big\| \cf^{-1} \widehat{g}(\cdot- \xi_k) \cf \vec{f} \big\|_{L^p(W)}^q  \bigg)^{1/q}.
\end{align}
We can assume that $\widehat{g}$ is a smooth bump function compactly supported in $B(0,100\sqrt{n})$ and $\widehat{g}(\xi)=1$ in $B(0,3\sqrt{n})$.
Then by Lemma \ref{wm-L1}, we have

\begin{align}\label{wm-2}
\| \square_k \vec{f} \|_{L^p(W)}
&=\| \cf^{-1} \phi_k \cf \vec{f} \|_{L^p(W)}=\| \cf^{-1} \phi_k \widehat{g}(\cdot- \xi_k) \cf \vec{f} \|_{L^p(W)} \nonumber \\
&\leq \| \cf^{-1}  \widehat{g}(\cdot- \xi_k) \cf \vec{f} \|_{L^p(W)}.
\end{align}
From (\ref{wm-1}) and (\ref{wm-2}), we immediately get that $\|f\|_{M^{s}_{p,q}(W)} \leq \|f\|^\circ_{M^{s}_{p,q}(W)}$.

For the reverse inequality, it is easy to see that the $\supp \widehat{g}(\cdot-\xi_k)$ intersects at most $O(\sqrt{n})$ many $\supp \phi_k$. by Lemma \ref{wm-L1}, we obtain
 \begin{align}\label{wm-3}
\big\| \cf^{-1} \widehat{g}(\cdot- \xi) \cf \vec{f} \big\|_{L^p(W)}
&= \bigg\| \cf^{-1} \sum_{ l \in \Lambda} \phi_{k+l} \widehat{g}(\cdot- \xi) \cf \vec{f} \bigg\|_{L^p(W)}  \nonumber\\
&\lesssim \sum_{ l \in \Lambda} \big\| \cf^{-1}  \phi_{k+l} \cf \vec{f} \big\|_{L^p(W)} \nonumber\\
& = \sum_{ l \in \Lambda} \|\square_{k + l} \vec{f}\|_{L^p(W)}.
\end{align}
 By (\ref{wm-1}) and (\ref{wm-3}), we have $\|\vec{f}\|^\circ_{M^{s}_{p,q}(W)} \lesssim \|\vec{f}\|_{M^{s}_{p,q}(W)}$. This completes the proof.
\end{proof}

\begin{thm}\label{w-eq1}
 Let $s \in \rr$, $1 \leq  p< \infty$, $0<q \leq \infty$ and $W$ be a $\mathcal{A}_p$ matrix. Also let $\vec{f} \in [\cs'(\rn)]^m$, then $\vec{f}\in M^{s}_{p,q}(W)$
 if and only if, there exists a sequence of continuous functions $\{\vec{f}_k\}_{k \in \zn} \subset L^p(W)$
 such that $\vec{f} =\sum_{k \in \zn} \square_k \vec{f}_k$ in $[\cs'(\rn)]^m$, $\|\{ \langle k \rangle^{s}\vec{f}_k\}_{k \in \zn}\|_{\ell^{q}(L^{p}(W))}<\infty$. In the case,
\[\|\vec{f}\|_{M^{s}_{p,q} (W)} \sim \inf \|\{\langle k \rangle^{s} \vec{f}_k\}_{k \in \zn}\|_{\ell^{q}(L^{p}(W))}.\]
where the infimum is taken for all above decompositions of $f$.
\end{thm}
\begin{proof}
Let $\vec{f} \in M^s_{p,q}(W)$. Put $\vec{f}_k= \sum_{l \in \Lambda}  \square_{k + l} \vec{f}$. Then, we obtain
\[ \vec{f} = \sum_{k \in \zn} \square_k \vec{f}_k \text{ \quad and } \|\{\langle k \rangle^{s} \vec{f}_k\}_{k \in \zn}\|_{\ell^{q}(L^{p}(W))} \lesssim \|f\|_{M^{s}_{p,q} (W)} .\]
On the other hand, by Lemma \ref{wm-L1}, we have
\[ \|\square_k \vec{f}\|_{L^{p}(W)} \lesssim \sum_{l \in \Lambda}  \|\square_k \square_{k+l} \vec{f}_{k+l}\|_{L^{p}(W)} \lesssim \sum_{l \in \Lambda}  \|\vec{f}_{k+l}\|_{L^{p}(W)} .\]
Thus, we have
\[ \|\vec{f}\|_{M^{s}_{p,q} (W)}  \lesssim \|\{\langle k \rangle^{s} \vec{f}_k\}_{k \in \zn}\|_{\ell^{q}(L^{p}(W))}.\]
This completes the proof.
\end{proof}

\section{Embeddings}\label{mwm-s4}
\begin{thm}\label{wm-em1}
Let $W$ be a $\mathcal{A}_p$ matrix, $s \in \rr$ and $1\leq p <\infty$.\\
{\rm (i)} If $ 0< q_0 \leq q_1 \leq \infty$, then
\[M^{s}_{p,q_0} (W) \hookrightarrow M^{s}_{p,q_1} (W);\]
{\rm (ii)} If $\epsilon>0$ and $\epsilon q_1>n$, then

\[M^{s+\epsilon}_{p,q_0} (W) \hookrightarrow M^s_{p,q_1} (W).\]
\end{thm}

\begin{proof}

{\rm (i)} From the  monotonicity  of  the $\ell^q$ space, we obtain the desired result.

{\rm (ii)} By {\rm (i)}, we have
\[ M^{s+\epsilon}_{p,q_0}(W)  \hookrightarrow B^{s+\epsilon,w}_{p,\infty}(W).\]
Since $\epsilon>0$, we obtain
\[\bigg( \sum_{k \in \zz^n} \langle k \rangle^{sq_1} \| W^{1/p} \cdot \square_k \vec{f}\|_{L^p}^{q_1} \bigg)^{1/q_1} \leq \sup_{k \in \zn}   \langle k \rangle^{s+\epsilon} \|W^{1/p} \cdot \square_k \vec{f}\|_{L^p} \bigg( \sum_{k \in \zz^n} \langle k \rangle^{-\epsilon q_1} \bigg)^{1/q_1}. \]
Observing that
\[\sum_{k \in \zz^n} \langle k \rangle^{-\epsilon q_1} \les \sum_{i=0}^\infty \langle i \rangle^{n-1-\epsilon q_1},\]
and the series on the right-hand side of the above inequality converges if $\epsilon q_1>n$. Thus, we have
 \[M^{s+\epsilon}_{p,\infty}(W)  \hookrightarrow M^{s}_{p,q_1}(W).\]
 Therefore, this completes the proof of {\rm (ii)}.
\end{proof}

\begin{lem}\label{w-em1}
Let $p \in [1,\infty)$ and $W\in \mathcal{A}_{p}.$ Then $ M^{0}_{p,1} (W) \hookrightarrow L^{p}(W ) \hookrightarrow M^{0}_{p,\infty}(W)$.
\end{lem}

\begin{proof}
 By Lemma \ref{wm-L1}, we obtain
\[ \|\vec{f}\|_{M^{0,w}_{p,\infty}} =\sup_{k \in \zn} \| \cf^{-1} \phi_k \cf \vec{f}\|_{L^{p}(W)} \leq C \| \vec{f}\|_{L^{p}(W)},\]
On the other hand, we estimate that
\[ \| \vec{f}\|_{L^{p}(W)} \leq \sum_{k \in \zn} \| \cf^{-1} \phi_k \cf \vec{f}\|_{L^{p}(W)}= \|\vec{f}\|_{M^0_{p,1}(W)}.\]
This completes the proof.
\end{proof}

\section{Duality}\label{mwm-s6}
In this section, we combine the results of Section \ref{mwm-s2} to give the duality of $M^{s}_{p,q}(W)$.
Let us define
\[ \ell^q_s (\zn, L^p(W)) := \big\{  \{\vec{f_k}\}_{k \in \zn} : \|\{\vec{f_k}\}_{k \in \zn}\|_{\ell^{q}(L^p(W))} < \infty \big\} ,\]
where
\[ \|\{\vec{f_k}\}_{k \in \zn}\|_{\ell^{q}(L^p(W))} = \bigg( \sum_{k \in \zn} \langle k \rangle^{sq} \|\vec{f}_k \|^q_{L^p(W)} \bigg)^{1/q} .\]

\begin{lem}[see {\cite[Proposition 2.3]{nm-1}}]\label{mwm-em3}
If $1\leq p <\infty$, $0 < q \leq \infty$, $s \in \rr$ and $W$ be a  $\mathcal{A}_p$ matrix, then
\[ [\cs(\rn)]^m \hookrightarrow M^{s}_{p,q}(W) \hookrightarrow  [\cs'(\rn)]^m.\]
Furthermore, $M^{s}_{p,q}(W)$ is a quasi-Banach space and $[\cs(\rn)]^m$  is dense in  $M^{s}_{p,q}(W)$.
\end{lem}

\begin{lem}\label{wm-L3}
Let $1<p,$ $q < \infty$, $s \in \rr$ and $W$ be a $\mathcal{A}_p$ matrix.  Then there exists a constant $C>0$ such that
\[\sum_{k \in \zn} \int_{\rn} |\vec{f}_k(x)\vec{g}_k(x)|{\rm d}x \leq
 C\|\{\vec{f}_k\}_{k \in \zn}\|_{\ell^{q}(L^{p}(\{A_Q\}))} \|\{\vec{g}_k\}_{k \in \zn}\|_{\ell^{q^{\prime}}(L^{p^{\prime}} (\{A_Q^{-1}\}))} \]
 holds for $\{\vec{f}_k\}_{k \in \zn}$ and $\{\vec{g}_k\}_{k \in \zn}$ be sequences of locally Lebesgue
integrable functions satisfying $\|\{\vec{f}_k\}_{k \in \zn}\|_{\ell^{q}(L^{p}(\{A_Q\}))} < \infty$
and $\|\{\vec{g}_k\}_{k \in \zn}\|_{\ell^{q^{\prime}}(L^{p^{\prime}} (\{A_Q^{-1}\}))} < \infty$.
\end{lem}

\begin{proof}
By the self-adjointness, the Cauchy-Schwarz inequality and H\"{o}lder's inequality, we have
\begin{align*}
\sum_{k \in \zn} \int_{\rn} |\vec{f}_k(x)\vec{g}_k(x)|{\rm d}x
&= \sum_{k \in \zn} \int_{\rn} |A_Q^{-1} A_Q \vec{f}_k(x)\vec{g}_k(x)|{\rm d}x \\
&\les \sum_{k \in \zn} \| \vec{f}_k\|_{L^p(\{A_Q\})} \|\vec{g}_k\|_{(L^{p^{\prime}}(\{A_Q^{-1}\}))}\\
&\les \|\{\vec{f}_k\}_{k \in \zn}\|_{\ell^{q}(L^{p}(\{A_Q\}))} \|\{\vec{g}_k\}_{k \in \zn}\|_{\ell^{q^{\prime}}(L^{p^{\prime}}(\{A_Q^{-1}\}))}.
\end{align*}
This completes the proof.
\end{proof}
\begin{prop}\label{wm-P1}
Let $1<p,$ $q < \infty$, $s \in \rr$ and $W$ be a $\mathcal{A}_p$ matrix. Then
\[ (\ell^q_s (\zn, L^p(\{A_Q\})))^{\prime} = \ell^{q^{\prime}}_{-s} (\zn, L^{p^{\prime}}(\{A_Q^{-1}\})).\]
Furthermore, $\{\vec{g_k}\}_{k \in \zn}\in  (\ell^q_s (\zn, L^p(\{A_Q\})))^{\prime}$  is equivalent to
\[ \langle \{\vec{g_k}\}_{k \in \zn},\{\vec{f_k}\}_{k \in \zn} \rangle = \sum_{k \in \zn} \int_{\rn} \vec{g}_k(x) \vec{f}_k(x) {\rm d}x \]
for all $\{\vec{f}_k\}_{k \in \zn} \in \ell^q_s (\zn, L^p(\{A_Q\})$, where
\[ \{\vec{g}_k\}_{k \in \zn} \in \ell^{q^{\prime}}_{-s} (\zn, L^{p^{\prime}}(\{A_Q^{-1}\})), \quad \|\vec{g} \|_{(\ell^q_s (\zn, L^p(\{A_Q\})))^{\prime}}  = \big\| \{\vec{g}_k\} \big\|_{\ell^{q^{\prime}}_{-s}(L^{p^{\prime}}(\{A_Q^{-1}\}))}.\]
\end{prop}
\begin{proof}
By Lemma \ref{wm-L3}, we have $(\ell^q_s (\zn, L^p(\{A_Q\})))^{\prime} \supset \ell^{q^{\prime}}_{-s} (\zn, L^{p^{\prime}}(\{A_Q^{-1}\}))$ and
\[ \langle \{\vec{g_k}\}_{k \in \zn}, \{\vec{f_k}\}_{k \in \zn} \rangle =\sum_{k \in \zn} \int_{\rn} \vec{g}_k(x) \vec{f}_k(x) {\rm d}x\]
for all $\{\vec{f}_k\}_{k \in \zn}  \in \ell^q_s (\zn, L^p(\{A_Q\})))$, $\{\vec{g}_k\}_{k\in \zn} \in \ell^{q^{\prime}}_{-s} (\zn, L^{p^{\prime}}(\{A_Q^{-1}\}))$.
For any $\vec{g}_k \in (\ell^q_s (\zn, L^p(\{A_Q\})))^{\prime}$, we denote
\[ \langle \vec{g}_k, \vec{f}_k \rangle = \langle \vec{g}, (0,0,\cdots, \vec{f}_k,0,0,\cdots) \rangle, \quad \{\vec{f}_k\}_{k \in \zn}  \in \ell^q_s (\zn, L^p(\{A_Q\}))). \]
It follows that $\vec{g}_k \in (L^p(\{A_Q\}))^{\prime}=L^{p^{\prime}}(\{A_Q^{-1}\})$, hence
\[ \langle \vec{g}_k, \vec{f}_k \rangle =   \int_{\rn} \vec{g}_k(x) \vec{f}_k(x){\rm d}x . \]
Thus,
\[ \langle \{\vec{g_k}\}_{k \in \zn}, \{\vec{f_k}\}_{k \in \zn} \rangle =  \sum_{k \in \zn} \int_{\rn} \vec{g}_k(x) \vec{f}_k(x){\rm d}x . \]
 Then $\vec{g}_k \in L^{p^{\prime}}(\{A_Q^{-1}\})$. We put $\vec{g}^{\prime}_k (x) = \vec{g}_k(x)/\lambda$,
where
\[\lambda = \big\|\{ \langle k \rangle^{-s} \vec{g}_k\}_{|k| \leq N}\big\|_{\ell^{q^{\prime}}(L^{p^{\prime}}(\{A_Q^{-1}\}))}.\]
 We set
\begin{align*}
\vec{f}_k&=\operatorname{sgn}\vec{g}_k \cdot | \langle k \rangle^{-sq^{\prime}} \vec{g}^{\prime}_k(x)^{q^{\prime}} A_Q^{-q^{\prime}}|^{\frac{p^{\prime}}{q^{\prime}} - \frac{1}{q^{\prime}}} \langle k \rangle^{-s} A_Q^{-1} \| \langle k \rangle^{-sq^{\prime}}\vec{g}^{\prime}_k(\cdot)^{q^{\prime}}A_Q^{-q^{\prime}}\|^{1-\frac{p^{\prime}}{q^{\prime}}}_{L^{\frac{p^{\prime}}{q^{\prime}}}}
\end{align*}
for each $k \in \zn$. Since
\begin{align*}
&\int_{\rn} \bigg(\frac{| \langle k \rangle^{sq} A_Q^{q} \vec{f}_k(x)^{q}|} {\| \langle k \rangle^{-sq^{\prime}} A_Q^{-q^{\prime}} \vec{g}^{\prime}_k(\cdot)^{q^{\prime}} \|_{L^{\frac{p^{\prime}}{q^{\prime}}}}}\bigg)^{\frac{p}{q}} {\rm d}x \\
&=\int_{\rn} \frac{| \langle k \rangle^{-sq^{\prime}} A_Q^{-q^{\prime}} \vec{g}^{\prime}_k(x)^{q^{\prime}} |^{\frac{p^{\prime}}{q^{\prime}}}
\| \langle k \rangle^{-sq^{\prime}}A_Q^{-q^{\prime}}\vec{g}^{\prime}_k(\cdot)^{q^{\prime}}\|^{\frac{p}{q}-\frac{p^{\prime}}{q^{\prime}}}_{L^{\frac{p^{\prime}}{q^{\prime}}}}} {\| \langle k \rangle^{-sq^{\prime}}A_Q^{-q^{\prime}}\vec{g}^{\prime}_k(\cdot)^{q^{\prime}} \|^{\frac{p}{q}}_{L^{\frac{p^{\prime}}{q^{\prime}}}}} {\rm d}x \\
&=\int_{\rn} \bigg(\frac{| \langle k \rangle^{-sq^{\prime}} \vec{g}^{\prime}_k(x)^{q^{\prime}} A_Q^{-q^{\prime}}|} {\| \langle k \rangle^{-sq^{\prime}}A_Q^{-q^{\prime}}\vec{g}^{\prime}_k(\cdot)^{q^{\prime}} \|_{L^{\frac{p^{\prime}}{q^{\prime}}}}}\bigg)^{\frac{p^{\prime}}{q^{\prime}}} {\rm d}x \\
&=1,
\end{align*}
so we have
\[\| \langle k \rangle^{sq} A_Q^{q} \vec{f}_k(\cdot)^{q}\|_{L^{\frac{p}{q}}}
= \|\langle k \rangle^{-sq^{\prime}}A_Q^{-q^{\prime}}\vec{g}^{\prime}_k(\cdot)^{q^{\prime}} \|_{L^{\frac{p^{\prime}}{q^{\prime}}}} .\]
Hence, we have
\[\|\{\langle k \rangle^{s} \vec{f}_k\}^N_{|k|=0}\|_{\ell^{q}(L^{p}(\{A_Q\}))} \leq 1.\]
We gain
\[\int_{\rn}\sum^N_{|k|=0} \vec{g}^{\prime}_k \vec{f}_k {\rm d}x =\sum^N_{|k|=0}\| \langle k \rangle^{-sq^{\prime}}A_Q^{-q^{\prime}}\vec{g}^{\prime}_k(\cdot)^{q^{\prime}}\|_{L^{\frac{p^{\prime}}{q^{\prime}}}}=1 .\]
Thus, we have
\[ \big\|\{ \langle k \rangle^{-s} \vec{g}_k\}^N_{|k|=0} \big\|_{\ell^{q^{\prime}}(L^{p^{\prime}}(\{A_Q^{-1}\})} \lesssim \|\{\vec{g_k}\}_{k \in \zn}\|_{(M^{s}_{p,q}(\{A_Q\}))^{\prime}} .\]
Therefore,
\[\{\vec{g_k}\}_{k \in \zn} \in M^{-s}_{p^{\prime},q^{\prime}}(\{A_Q^{-1}\}) \text{ and }  \|\{\vec{g_k}\}_{k \in \zn}\|_{M^{-s}_{p^{\prime},q^{\prime}}(\{A_Q^{-1}\})}
\lesssim \|\{\vec{g_k}\}_{k \in \zn}\|_{(M^{s}_{p,q} (\{A_Q\}))^{\prime}} .\]
This completes the proof.
\end{proof}

\begin{thm}\label{wm-du1}
Let $1<p,$ $q < \infty$, $s \in \rr$ and $W$ be a $\mathcal{A}_p$ matrix. Then
\begin{equation}\label{mwm-6}
\big(M^{s}_{p,q}(W) \big)^{\prime} =  M^{-s}_{p^{\prime},q^{\prime}}(W^{-p^{\prime}/p}).
\end{equation}
\end{thm}
\begin{proof}
We want to show (\ref{mwm-6}), let $\{A_Q\}_{Q \in \mathcal{Q}}$ be a sequence of reducing operators of order $p$ for $W$,  and we can see from Lemma \ref{mwm-L2} that we just need to show
\[\big(M^{s}_{p,q}(\{A_Q\}) \big)^{\prime} \sim M^{-s}_{p^{\prime},q^{\prime}}(\{A_Q^{-1}\}).\]
Step 1. Let $\vec{f} \in [\mathcal{S}(\rn)]^m$ and $\vec{g} \in M^{-s}_{p^{\prime},q^{\prime}}(A^{-1}_Q)$. By the self-adjointness, the Cauchy-Schwarz inequality and the H\"{o}lder inequality, we have
\begin{align*}
  \langle \vec{g},\vec{f} \rangle
  &=\sum_{k \in \zn} \sum_{l \in \Lambda}  \langle \square^{\ast}_{k+l} \vec{g},\square_k \vec{f} \rangle \\
   &=\sum_{k \in \zn} \sum_{l \in \Lambda} \langle k\rangle^s \langle k\rangle^{-s} \int_{\rn} A_Q A_Q^{-1}  \square_{k+l}^\ast \vec{g}\square_k \vec{f} {\rm d}x \\
   & \les \sum_{k \in \zn} \sum_{l \in \Lambda} \langle k\rangle^{-s} \|\square_{k+l}^\ast \vec{g} \|_{L^{p^{\prime}}(\{A_Q^{-1}\})} \langle k \rangle^s \|\square_{k} \vec{f} \|_{L^p(\{A_Q\})}\\
   & \les  \|\vec{g}\|_{M^{-s}_{p^{\prime},q^{\prime}}(\{A_Q^{-1}\})} \|\vec{f}\|_{M^{s}_{p,q}(\{A_Q\})},
\end{align*}
where $\square^{\ast}_{k} =\cf \phi_k \cf^{-1}$. Since $[\mathcal{S}(\rn)]^m$ is dense in $M^{s}_{p,q}(W)$ (Lemma \ref{mwm-em3}), hence
\[\vec{g} \in (M^{s}_{p,q}(\{A_Q\}))^{\prime}, \ \|\vec{g}\|_{(M^{s}_{p,q}(\{A_Q\}))^{\prime}} \leq C\|\vec{g}\|_{M^{-s}_{p^{\prime},q^{\prime}}(\{A_Q^{-1}\})} .\]

We prove  $\big(M^{s}_{p,q}(\{A_Q\}) \big)^{\prime} \subset M^{-s}_{p^{\prime},q^{\prime}}(\{A_Q^{-1}\})$.  It is easy to see that, for $\{\vec{f_k}\}_{k \in \zn} \in M^{s}_{p,q}(\{A_Q\})$, the map: $\{\vec{f_k}\}_{k \in \zn} \mapsto \{\square_k \vec{f}\}_{k \in \zn} \in \ell^q_s(L^p(\{A_Q\}))$ is isometric from $M^{s}_{p,q}(\{A_Q\})$ into the subspace $X$ of $\ell^p_s(\zn,L^p(\{A_Q\}))$.
By Proposition \ref{wm-P1}, for all $\{\vec{f}_k\}_{k \in \zn} \in \ell^q_s(\zn, L^p(\{A_Q\}))$, we have
\[ \langle \{\vec{g_k}\}_{k \in \zn},\{\vec{f_k}\}_{k \in \zn} \rangle=\sum_{k \in \zn} \int_{\rn} \vec{g_k}(x) \vec{f_k}(x) {\rm d}x ,\]
where $\{\vec{g_k}\}_{k \in \zn} \in \ell^{q^{\prime}}_{-s} (\zn, L^{p^{\prime}}(\{A_Q^{-1}\}))$ and
\[ \|\{\vec{g_k}\}_{k \in \zn}\|_{(M^{s}_{p,q}(\{A_Q\}))^{\prime}} =  \big\| \{\vec{g}_k\}_{k \in \zn} \big\|_{\ell^{q^{\prime}} (L^{p^{\prime}}_{-s}(\{A_Q^{-1}\})} .\]
Since $\{ \square^{\ast}_k \phi\}_{k \in \zn} \in \ell^q_s (\zn, L^p(\{A_Q\}))$ for any $\varphi \in \cs(\rn)$, we have
\[ \langle \{\vec{g_k}\}_{k \in \zn},\{ \square^{\ast}_k \varphi\}_{k \in \zn} \rangle = \sum_{k \in \zn} \int_{\rn} \vec{g}_k(x)  \square^{\ast}_k \varphi(x) {\rm d}x = \int_{\rn} \sum_{k \in \zn}  \square_k \vec{g}_k(x) \varphi(x) {\rm d}x, \]
which implies $\vec{g}=\sum_{k \in \zn} \square_k \vec{g}_k(x)$. Therefore, by Theorem \ref{w-eq1}, we obtain
\[  \|\vec{g}\|_{M^{-s}_{p^{\prime},q^{\prime}}(\{A_Q^{-1}\})}  \lesssim \big\| \{\vec{g}_k\}_{k \in \zn} \big\|_{\ell^{q^{\prime}} (L^{p^{\prime}}_{-s}(\{A_Q^{-1}\})}  \lesssim \|\vec{g}\|_{(M^{s}_{p,q}(\{A_Q\}))^{\prime}} .  \]
This completes the proof.
\end{proof}

\section*{Competing Interests} The authors declare that there is no conflict of interest.

\end{document}